\documentclass{article}
\usepackage{amssymb,amsmath,amsthm,graphicx,cite}

\def\qed{\hfill {\hbox{${\vcenter{\vbox{               
   \hrule height 0.4pt\hbox{\vrule width 0.4pt height 6pt
   \kern5pt\vrule width 0.4pt}\hrule height 0.4pt}}}$}}}

\def\tr{\triangleright}

\newtheorem{theorem}{Theorem}

\newtheorem{corollary}[theorem]{Corollary}

\theoremstyle{definition}
\newtheorem{example}{Example}
\newtheorem{definition}{Definition}
\newtheorem{remark}{Remark}

\date{}

\title{\Large \textbf{Quandle Module Quivers}}

\author{
Karma Istanbouli\footnote{Email: kistanbo3308@scrippscollege.edu}
 \and
Sam Nelson\footnote{Email: Sam.Nelson@cmc.edu.}}

\begin{document}
\maketitle

\begin{abstract}
We enhance the quandle coloring quiver invariant of oriented knots and links
with quandle modules. This results in a two-variable polynomial invariant
with specializes to the previous quandle module polynomial invariant
as well as to the quandle counting invariant. We provide example computations
to show that the enhancement is proper in the sense that it distinguishes 
knots and links with the same quandle module polynomial.
\end{abstract}

\parbox{6in} {\textsc{Keywords:} Quandle Quivers, Quandle Modules, Enhancements
of quandle counting invariants

\smallskip

}

\section{\large\textbf{Introduction}}\label{I}

\textit{Quandles} are algebraic structures with axioms corresponding to the 
Reidemeister moves for oriented knots and links. In \cite{J} and \cite{M}, 
a quandle
is associated to each such knot or link $L$, known as the \textit{knot 
quandle}; let us denote this by $Q(L)$. Given a finite quandle $X$, the set of
quandle homomorphisms $\mathrm{Hom}(Q(L),X)$ from the knot quandle to $X$
is an invariant of oriented links (To be precise, the abstract set of 
homomorphisms does not change when we do Reidemsiter moves, though 
our practical ways of writing them change with diagrams analogously to 
writing the same linear transformation with respect to different bases).
The elements of this set can be visualized
as \textit{colorings} of a diagram $D$ of $L$ by $X$, i.e., assignments of
elements of $X$ to the arcs in $D$ satisfying the \textit{quandle coloring
condition}
at every crossing. 
The number of such colorings, denoted by
$\Phi_X^{\mathbb{Z}}(L)=|\mathrm{Hom}(Q(L),X)|$, is an integer-valued oriented 
link invariant known as the \textit{quandle counting invariant.}

Any invariant $\phi$ of quandle colored knots and links defines an invariant
of uncolored oriented knots and links by taking the multiset of $\phi$-values 
over the set of quandle colorings of our knot or link. For large multisets
it can be helpful to encode the multiset as a polynomial by writing the elements
as powers of a formal variable $u$ with multiplicities as coefficients; this
notation has the advantage that evaluation at $u=1$ yields the quandle counting
invariant. This type of invariant is known as an \textit{enhancement}; see
\cite{EN} and the references therein (especially \cite{J,M}) for more.

A type of enhancement known as \textit{Quandle coloring quivers} were 
introduced in \cite{CN} by the second listed author and coauthor Karina Cho. 
Given a set of endomorphisms 
$f\in\mathrm{End}(X)$ of a finite quandle $X$, a quiver-valued invariant
$\mathcal{Q}_X(L)$ of oriented knots and links $L$ is defined, categorifying
the quandle counting invariant.

In \cite{AG} the notion of \textit{rack modules} with coefficients
in a commutative ring, including the special case of \textit{quandle modules},
was introduced. In later work such as \cite{CEGS,HHNYZ} quandle modules were
used to enhance the quandle counting invariant, obtaining a one-variable
polynomial invariant of oriented knots and links known as the \textit{quandle 
module polynomial} which evaluates to the quandle counting invariant at $u=1$.

In this paper we define \textit{quandle module quivers}, using quandle modules
to enhance the quandle coloring quiver analogously to the use of quandle 
cocycles in \cite{CN2}. As a result we obtain a new two-variable polynomial
which specializes to the quandle module polynomial but can distinguish links 
with the same quandle module polynomial value.

The paper is organized as follows. In Section \ref{B} we review the basics
of quandles, quandle modules and quandle coloring quivers. In Section 
\ref{QMQ} we define quandle module quivers and compute some examples. In 
particular, we show by example that the new invariant is stronger than the 
previous invariant. We conclude in Section \ref{QQQ} with some questions for 
future work.

\section{\large\textbf{Quandles, Quandle Modules and Quivers}}\label{B}

In this section we review the basics of quandles, quandle modules and quandle
coloring quiver. We begin with a definition (see \cite{EN} and the references 
therein).

\begin{definition} 
A \textit{quandle} is a set $X$ with an operation $\tr$ satisfying the 
conditions 
\begin{itemize}
\item[(i)] For all $x\in X$, $x\tr x=x$,
\item[(ii)] For all $y\in X$, the map $f_y:X\to X$ defined by $f_y(X)=x\tr y$ is
invertible, and
\item[(iii)] For all $x,y,z\in X$ we have
\[(x\tr y)\tr z= (x\tr z)\tr (y\tr z).\]
If $f_y^{-1}=f_y$, i.e., if $(x\tr y)\tr y=x$ for all $x,y\in X$, then we say
$X$ is a \textit{kei} or \textit{involutory quandle}.
\end{itemize} See \cite{J,M} for more.
\end{definition}

\begin{example}
Examples of quandles include
\begin{itemize}
\item Every group is a kei under the \textit{core operation} $x\tr y=yx^{-1}y$,
\item Every group is a quandle under \textit{conjugation} $x\tr y=y^{-1}xy$,
\item Every $\mathbb{Z}[t^{\pm 1}]$-module is a quandle (known as an \textit{Alexander quandle} under $x\tr y=tx+(1-t)y$,
\item Every finite set $X=\{1,2,\dots, n\}$ can be given a quandle structure via an operation table satisfying the quandle axioms,
\end{itemize}
and many more such structures.
\end{example}

\begin{definition}
Let $K$ be an oriented link diagram and $X$ a quandle. An assignment of 
elements of $X$ to the arcs in $K$ is a \textit{quandle coloring of $K$ by $X$}
or an \textit{$X$-colorings of $K$} if at every crossing in $K$ we have
\[\includegraphics{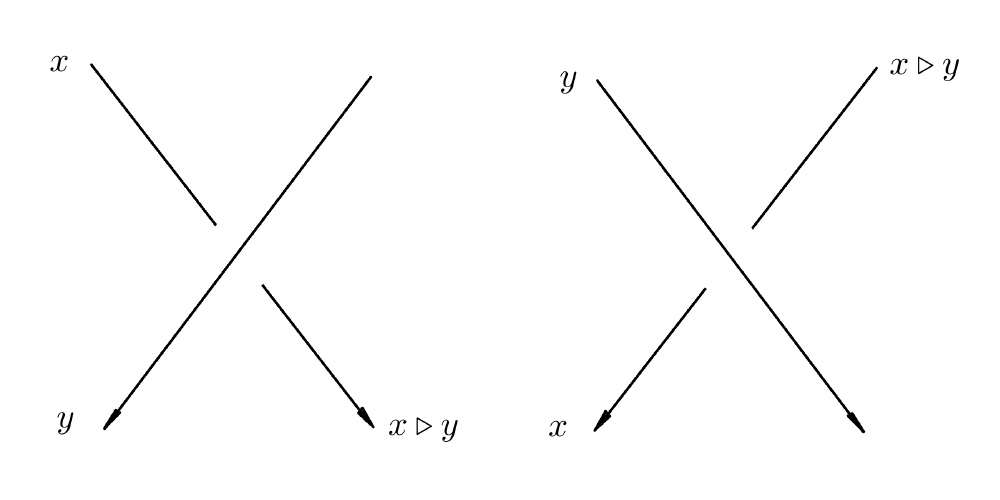}.\]
We will denote the set of $X$-colorings of $L$ by $\mathcal{C}(L,X)$.
\end{definition}

The quandle axioms are chosen so that for every oriented Reidemeister move,
$X$-colorings on one side of the move correspond bijectively with $X$-colorings
on the other side of the move (see \cite{EN} for more details). Thus we have:

\begin{theorem} (See \cite{J,M} etc.)
The number of $X$-colorings of an oriented knot or link diagram $K$ is an 
invariant of oriented knots and links, known as the \textit{quandle counting 
invariant}, denoted $\Phi_X^{\mathbb{Z}}(K)=|\mathcal{C}(L,X)|$.
\end{theorem}

\begin{example}\label{ex:1}
Let $X=\{1,2,3,4\}$ have the quandle structure given by
\begin{equation}\begin{array}{r|rrrr}
\tr & 1 & 2 & 3 & 4 \\ \hline
1 & 1 & 3 & 4 & 2 \\
2 & 4 & 2 & 1 & 3 \\
3 & 2 & 4 & 3 & 1 \\
4 & 3 & 1 & 2 & 4.
\end{array}\label{q1}
\end{equation}
The reader can verify that the $(4,2)$-torus link $L$ has sixteen $X$-colorings
\[\includegraphics{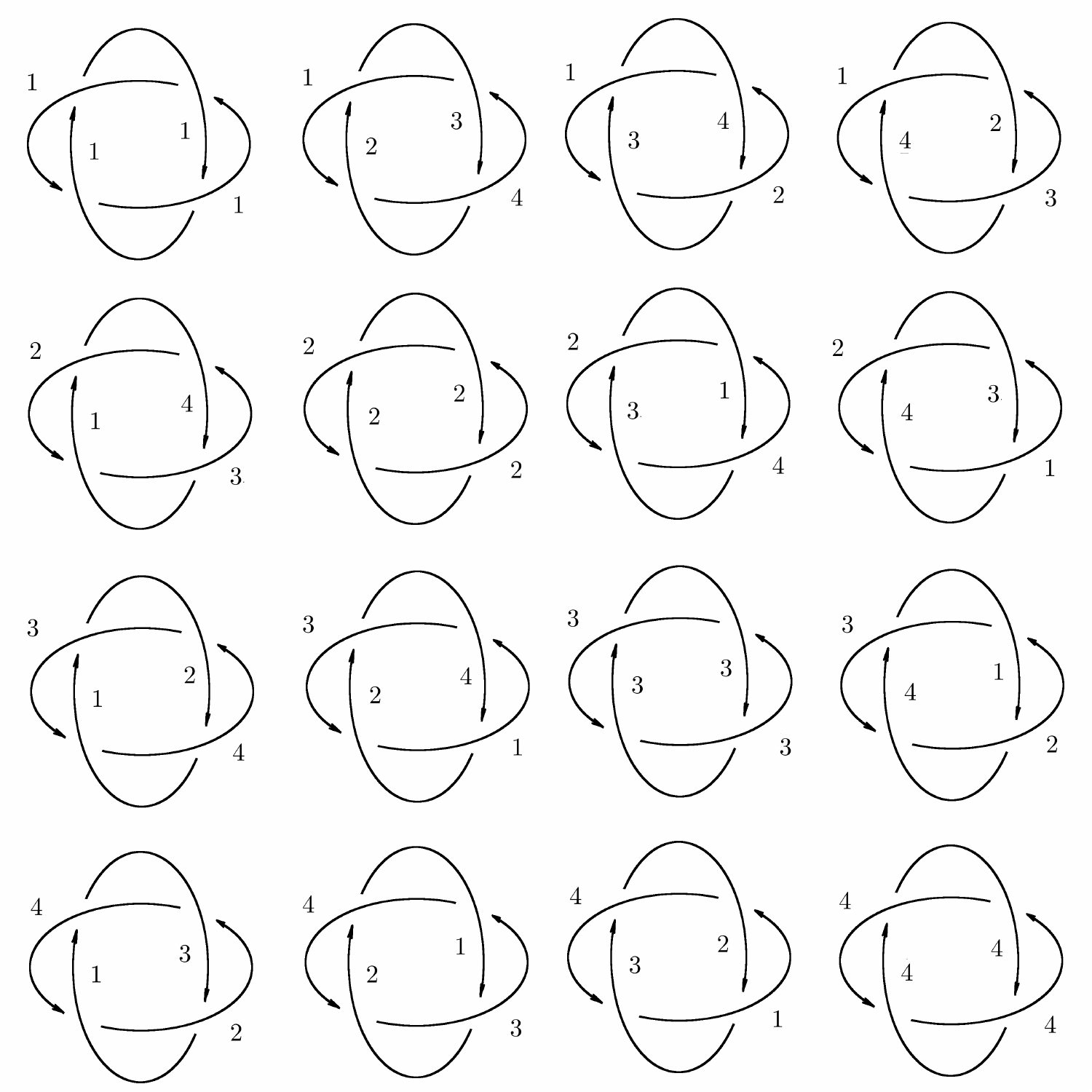}\]
while the Hopf link $L'$ has only four
\[\includegraphics{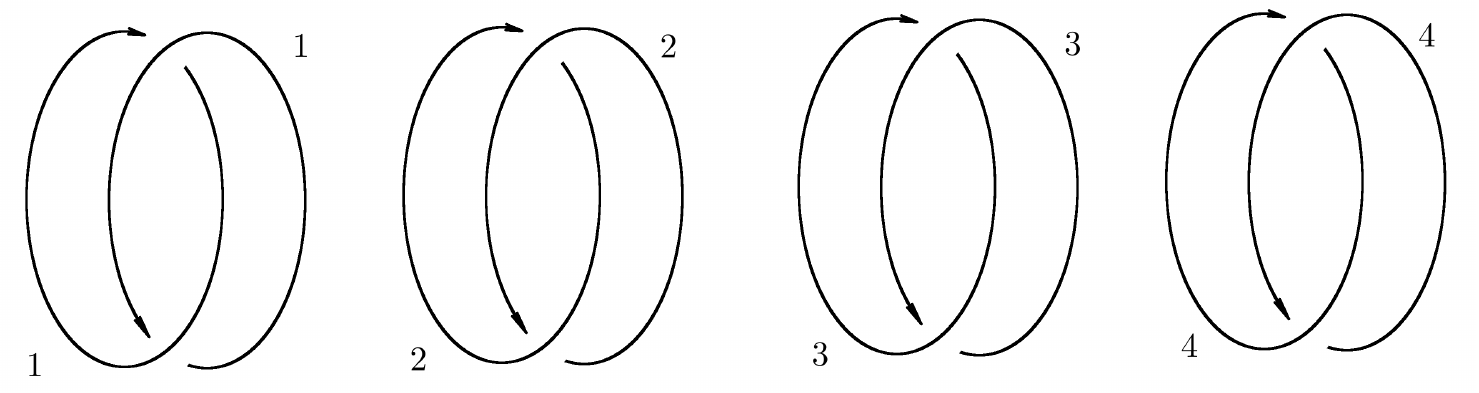}.\]
Hence, $\Phi_{X}^{\mathbb{Z}}(L)\ne \Phi_{X}^{\mathbb{Z}}(L')$ and the
counting invariant distinguishes the links.
\end{example}

Next, we recall quandle coloring quivers; see \cite{CN}.
Recall that a \textit{quandle endomorphism} is a function
$f:X\to X$ such that for all $x,y\in X$ we have
\[f(x\tr y)=f(x)\tr f(y).\]
Then for every $X$-coloring of an oriented link diagram $L$, applying $f$
to each color yields another $X$-coloring. Then for any oriented link
$L$, the \textit{quandle coloring quiver}
$\mathcal{Q}^f_X(L)$
associated to the quandle endomorphism $f$ is the directed graph with
one vertex for each coloring and a directed edge from each vertex to
the vertex associated to the coloring obtained by applying $f$ to each color.

\begin{example}\label{ex:2}
Let $X$ be the quandle (\ref{q1}) from example \ref{ex:1}.
The reader can verify that the map $f:X\to X$ given by
$[f(1),f(2),f(3),f(4)]=[2,4,3,1]$ is an endomorphism. Then the quandle 
coloring quiver $\mathcal{Q}_X^f(L)$ for the (4,2)-torus link $L$ looks like
\[\scalebox{1.25}{\includegraphics{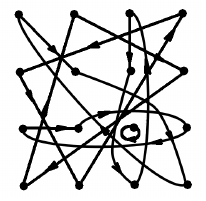}}\] 
or drawn more symmetrically,
\[\includegraphics{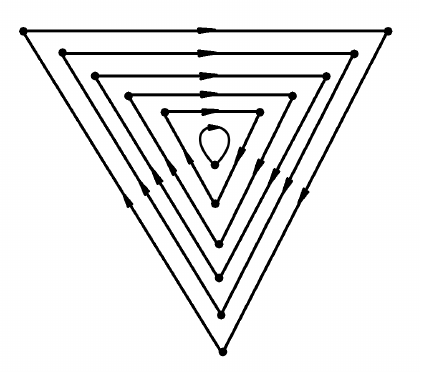}.\]
\end{example}

Since the set $\mathrm{Hom}(Q(L),X)$ is an invariant of oriented links, for
any $f\in\mathrm{End}(X)$ the quandle coloring quiver $\mathcal{Q}_X^f(L)$
is also an invariant. To be precise, the abstract homset 
$\mathrm{Hom}(Q(L),X)$ and endomorphism $f:X\to X$ together determine 
$\mathcal{Q}_X^f(L)$, and neither of these is dependent on our choice of 
diagram for $L$. In \cite{CN} examples are given which show
that knots and links with the same number of quandle colorings by a
chosen coloring quandle $X$ can be distinguished by the isomorphism type
of their quandle coloring quivers. Moreover, since a quiver is a category,
$\mathcal{Q}_X^f(L)$ is a categorification of the quandle counting invariant
with decategorification given by counting vertices in the quiver.

Now, let $X$ be a quandle and let
$R$ be a commutative ring with identity and let $t,s:X^2\to R$
be maps from the set of pairs of elements of $X$ to $R$.

\begin{definition} (See \cite{HHNYZ}.)
We say that $t,s$ define a \textit{quandle module} or \textit{$X$-module}
with coefficients in $R$ if for all $x,y,z\in X$ we have $t_{x,y}\in R^{\times}$ 
and
\[\begin{array}{rcl}
t_{x,x}+s_{x,x} & = & 1\\
t_{x\tr y,z}t_{x,y} & = & t_{x\tr z, y\tr z} t_{x,z} \\
t_{x\tr y, z}s_{x,y} & = & s_{x\tr y, z}t_{y,z} \\
s_{x\tr y, z} & = & t_{x\tr z, y\tr z} s_{x,z} +s_{x\tr z, y\tr z} s_{y,z}.
\end{array}\]

\end{definition}

Given a quandle module $s,t$ and a quandle-colored tame link diagram $L_C$, we
obtain the \textit{quandle module matrix} $M_{t,s}(L_c)$ representing the
homogeneous system of crossing equations over $R$ determined by $L_C$
according to the rule
\[\includegraphics{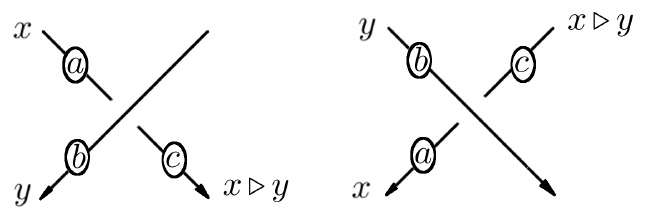} \quad \raisebox{0.4in}{$c=t_{x,y}a+s_{x,y}b$}.\]
Here the circled $a,b,c$ are ``beads'' assigned to each arc in the link
diagram, representing elements of $R$.

\begin{example}\label{ex:qminv}
The figure eight knot $L=4_1$ has 16 colorings by the quandle $X$ with 
operation table
\[\begin{array}{r|rrrr}
\tr & 1 & 2 & 3 & 4 \\ \hline
1 & 1 & 4 & 2 & 3\\
2 & 3 & 2 & 4 & 1 \\
3 & 4 & 1 & 3 & 2\\
4 & 2 & 3 & 1 & 4
\end{array}\]
including for instance 
\[L_1=\quad\raisebox{-0.5in}{\includegraphics{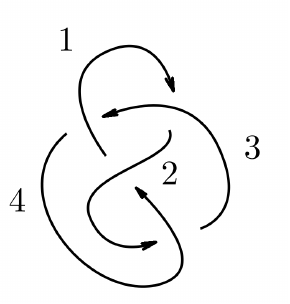}}.\]
Placing a bead on each arc, we have the diagram:
\[\includegraphics{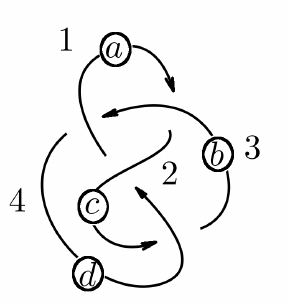}.\]
The quandle module $t,s$ over $X$ with $R=\mathbb{Z}_5$ coefficients
\[
\begin{array}{r|rrrr}
t & 1 & 2 & 3 & 4 \\\hline
1 & 4 & 2 & 4 & 1\\
2 & 1 & 4 & 2 & 4\\
3 & 3 & 1 & 4 & 1\\
4 & 3 & 2 & 3 & 4
\end{array}
\quad\quad
\begin{array}{r|rrrr}
s & 1 & 2 & 3 & 4 \\ \hline
1 & 2 & 1 & 3 & 1 \\
2 & 3 & 2 & 4 & 4 \\
3 & 1 & 2 & 2 & 4 \\
4 & 2 & 3 & 3 & 2
\end{array}
\]
then determines quandle module matrix
\[
M_{t,s}(L_1)=\left[\begin{array}{rrrr}
t_{13} & s_{13} & -1 & 0 \\
s_{31} & t_{31} & 0 & -1 \\
t_{12} & 0 & s_{12} & -1 \\
0 & t_{34} & -1 & s_{34} \\
\end{array}\right]
=\left[\begin{array}{rrrr}
4 & 3 & 4 & 0 \\
1 & 3 & 0 & 4 \\
2 & 0 & 1 & 4 \\
0 & 1 & 4 & 4 \\
\end{array}\right]
\]
which after row-reduction mod 5 becomes
\[\left[\begin{array}{rrrr}
1 & 3 & 0 & 4 \\
0 & 1 & 4 & 4 \\
0 & 0 & 0 & 0 \\
0 & 0 & 0 & 0 \\
\end{array}\right].\]
Thus this quandle coloring of this knot has quandle module matrix with 
kernel of dimension $2$ and hence $5^2=25$ bead colorings.
\end{example}

By construction (see also \cite{EN} and references therein) we have the 
following result:

\begin{theorem}
Let $X$ be a finite quandle, $R$ a commutative ring with identity, and 
$s,t:X^2\to R$ a quandle module. Then for any oriented link $L$, we define
a multiset-valued invariant of oriented links as follows:
\begin{itemize}
\item If $R$ is finite, we define 
\[\Phi^{M,t,s}_X(L)=\{|\mathrm{Ker} M_{t,s}(L_C) |\ |\ L_C\in\mathcal{C}(L,X)\}\]
and
\item If $R$ is infinite, we define
\[\Phi^{M,t,s}_X(L)=\{\mathrm{Rank}(\mathrm{Ker} M_{t,s}(L_C) )\ |\ L_C\in\mathcal{C}(L,X)\}\]
where $\mathrm{Rank}(S)$ is the number of generators of the free part of $S$
considered as an $R$-module. 
\end{itemize}
Then $\Phi^{M,t,s}_X$ is an invariant of oriented links.
\end{theorem}

\begin{remark}
For convenience we can encode this
multiset as a polynomial in a formal variable $\sigma$ known as the 
\textit{quandle module polynomial invariant},
\[\Phi^{t,s}_X(L)=\sum_{L_c\in\mathcal{C}(L,X)}\sigma^{|\mathrm{Ker} M_{t,s}(L_c)|}\]
if $R$ is finite or 
\[\Phi^{t,s}_X(L)=\sum_{L_c\in\mathcal{C}(L,X)}\sigma^{\mathrm{Rank}(\mathrm{Ker} M_{t,s}(L_c))}\]
if $R$ is infinite.
\end{remark}

\begin{example}
Let $X$, $R$ and $t,s$ be as in Example \ref{ex:qminv} above. Repeating the 
computation of bead colorings for each coloring, we obtain the multiset-valued
invariant value 
\[\Phi^{M,t,s}_X(4_1)=\{
25,25,25,25,
25,25,25,25,
25,25,25,25,
25,25,25,25\}\]
or in polynomial form, $\Phi^{t,s}_X(4_1)=16\sigma^{25}$.
The trefoil knot $3_1$ also has 16 $X$-colorings by this quandle and
thus is not distinguished from the figure eight by the
counting invariant with respect to this particular quandle; however,
$\Phi^{t,s}_X(3_1)=16\sigma^{5}$ and the quandle module enhancement detects
the difference between $3_1$ and $4_1$.
\end{example}

\section{\large\textbf{Quandle Module Quivers}}\label{QMQ}

We can now state our primary new definition.

\begin{definition}
Let $X$ be a finite quandle, $S\subset \mathrm{End}(X)$ a set of endomorphisms,
$R$ a commutative ring with identity and $(t,s)$ an $X$-module with
coefficients in $R$. Then for any oriented link $L$, the 
\textit{quandle module quiver} of $L$ with respect to the above objects
is the quiver $\mathcal{Q}_{X,S}^{t,s}(L)$ obtained by assigning a weight of
$|\mathrm{Ker}M_{t,s}(L_c)|$ 
(or $\mathrm{Rank}(\mathrm{Ker}M_{t,s}(L_c))$ if $R$ is infinite)
to each vertex in the quandle coloring quiver.
When $S=\{f\}$ is a singleton, we will write $\mathcal{Q}_{X,f}^{t,s}(L)$
instead of $\mathcal{Q}_{X,\{f\}}^{t,s}(L)$ for simplicity.
\end{definition}

We have immediately our main result:

\begin{theorem}\label{thm:main}
For any quandle $X$, set of endomorphisms $S\subset \mathrm{End}(X)$, 
ring $R$, and quandle module $(t,s)$, the weighted quiver 
$\mathcal{Q}_{X,S}^{t,s}(L)$ is an invariant of oriented links.
\end{theorem}

\begin{proof}
All of the data required to construct $\mathcal{Q}_{X,S}^{t,s}(L)$ is already
known to be invariant under Reidemeister moves; hence there is nothing to show.
\end{proof}

\begin{example}\label{Ex:qmex}
Taking the same quandle and endomorphism from Example \ref{ex:2} with 
quandle module over $X$ with coefficients in $\mathbb{Z}_4$ given by
\[
\begin{array}{r|rrrr}
t & 1 & 2 & 3 & 4 \\ \hline
1 & 3 & 1 & 1 & 1 \\
2 & 3 & 3 & 3 & 1 \\
3 & 3 & 1 & 3 & 3 \\
4 & 3 & 3 & 1 & 3
\end{array}
\quad\quad
\begin{array}{r|rrrr}
s & 1 & 2 & 3 & 4 \\ \hline
1 & 2 & 2 & 2 & 2\\
2 & 2 & 2 & 2 & 2\\
3 & 2 & 2 & 2 & 2\\
4 & 2 & 2 & 2 & 2
\end{array}\]
we obtain quandle module quiver
\[\includegraphics{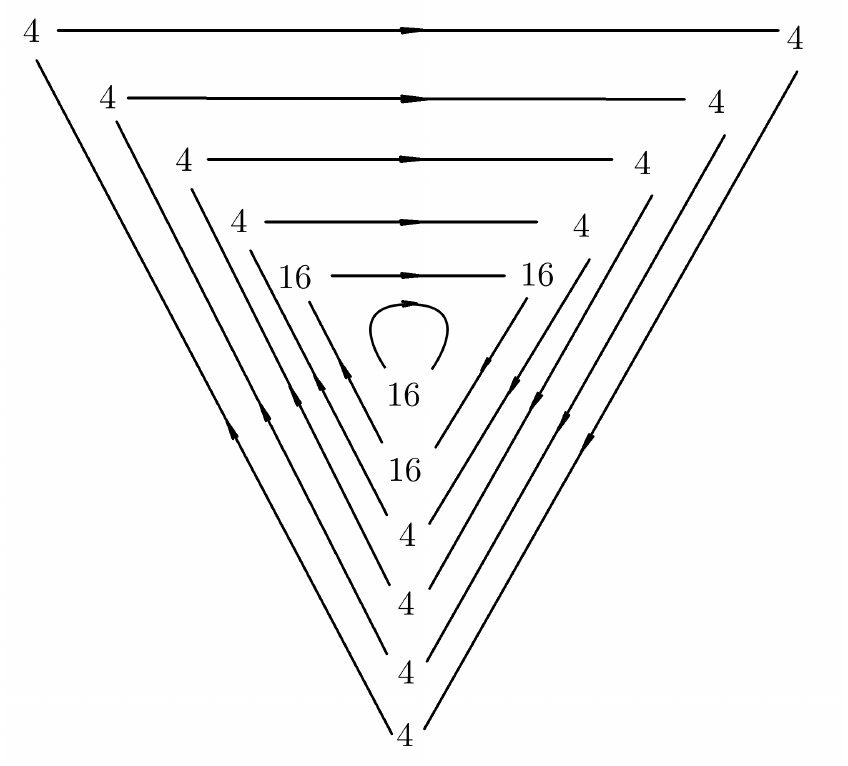}\label{qmex}\]
for the $(4,2)$-torus link.
\end{example}

As with other quandle quiver invariants, for small quivers it is easy enough
to compare the quivers directly, but for large quivers it can be useful to 
extract invariants which are easier to compare. To this end we define a
two-variable polynomial we call the \textit{quandle module quiver polynomial}:

\begin{definition}
Let $X$ be a finite quandle, $S\subset\mathrm{End}(X)$ a set of endomorphisms, 
$R$ a commutative ring with identity and $t,s:X\times X\to R$ a quandle
module with coefficients in $R$. We define the \textit{quandle module
quiver polynomial}, denoted $\Phi_{X,S}^{t,s}(L)$, to be the sum 
over the set of edges in the quandle module quiver of 
$\sigma^{|\mathrm{Ker}M_{t,s}(L_1)|}\tau^{|\mathrm{Ker}M_{t,s}(L_2)|}$
( or $\sigma^{\mathrm{Rank}(\mathrm{Ker}M_{t,s}(L_1))}\tau^{\mathrm{Rank}(\mathrm{Ker}M_{t,s}(L_2))}$ if $R$ is infinite)
where the edge is directed from the vertex corresponding to coloring $L_1$
to the vertex corresponding to coloring $L_2$. That is, we set
\[\Phi_{X,S}^{t,s}(L)=\sum_{e\in E(\mathcal{Q}_{X,S}^{t,s}(L))} \sigma^{w(h(e))}\tau^{w(t(e))}
\]
where $w(h(e))$ and $w(t(e))$ are weights at the head and tail of $e$ 
respectively.
\end{definition}

\begin{example}
The quandle module quiver polynomial for the quandle module quiver \ref{qmex} 
in Example\ref{Ex:qmex} is 
\[12\sigma^4\tau^4+4\sigma^{16}\tau^{16}.\]
\end{example}

Theorem \ref{thm:main} implies the following:
\begin{corollary}
$\Phi_{X,S}^{t,s}(L)$ is an invariant of oriented knots and links.
\end{corollary}

\begin{example}
Let $X$ be the quandle with operation table
\[
\begin{array}{r|rrrrr}
\tr & 1 & 2 & 3 & 4 & 5 \\ \hline
  1 & 1 & 1 & 2 & 2 & 2\\
  2 & 2 & 2 & 1 & 1 & 1\\
  3 & 3 & 3 & 3 & 3 & 3\\
  4 & 5 & 5 & 4 & 4 & 4 \\
  5 & 4 & 4 & 5 & 5 & 5
\end{array}
\]
and let $f$ be the quandle endomorphism $f=[3,3,4,3,3]$.
Our \texttt{python} computations reveal that
$X$ has quandle modules with $\mathbb{Z}_3$ coefficients including
\[
\begin{array}{r|rrrrr}
t & 1 & 2 & 3 & 4 & 5 \\ \hline
1 & 1 & 1 & 1 & 2 & 2\\
2 & 1 & 1 & 1 & 2 & 2\\
3 & 1 & 1 & 1 & 2 & 2\\
4 & 2 & 2 & 1 & 2 & 2\\
5 & 2 & 2 & 1 & 2 & 2\\
\end{array}
\quad\quad
\begin{array}{r|rrrrr}
s & 1 & 2 & 3 & 4 & 5 \\ \hline
1 & 0 & 0 & 0 & 0 & 0\\
2 & 0 & 0 & 0 & 0 & 0\\
3 & 1 & 1 & 0 & 2 & 1\\
4 & 2 & 2 & 0 & 2 & 1\\
5 & 1 & 1 & 0 & 1 & 2
\end{array}
\]
giving us the following invariant values for the links
with up to seven crossings from the Thistlethwaite link table
found at \cite{KA}:
\[
\begin{array}{r|l} 
L & \Phi_{X,f}^{\tau,s}(L) \\ \hline
L2a1 & 4\sigma^9\tau^9 + \sigma^9\tau^3 + 4\sigma^3\tau^9 + 4\sigma^3\tau^3 \\
L4a1 & 16\sigma^9\tau^9 + \sigma^9\tau^3 + 8\sigma^3\tau^9 \\
L5a1 & 20\sigma^9\tau^9 + \sigma^9\tau^3 + 4\sigma^3\tau^9 \\
L6a1 & 21\sigma^9\tau^9 + 4\sigma^3\tau^9 \\
L6a2 & 4\sigma^9\tau^9 + \sigma^9\tau^3 + 4\sigma^3\tau^9 + 4\sigma^3\tau^3 \\
L6a3 & 9\sigma^9\tau^9 + 4\sigma^3\tau^3 \\
L6a4 & 20\sigma^{27}\tau^{27} + 15\sigma^{27}\tau^9 + 40\sigma^9\tau^{27} \\
L6a5 & 8\sigma^{27}\tau^{27} + \sigma^{27}\tau^9 + 8\sigma^9\tau^{27} + 6\sigma^9\tau^9 + 12\sigma^9\tau^3 + 6\sigma^3\tau^9 \\
L6n1 & 8\sigma^{27}\tau^{27} + \sigma^{27}\tau^3 + 6\sigma^9\tau^9 + 12\sigma^9\tau^3 + 8\sigma^3\tau^{27} + 6\sigma^3\tau^9 \\
L7a1 & 25\sigma^9\tau^9 \\
L7a2 & 16\sigma^9\tau^9 + \sigma^9\tau^3 + 8\sigma^3\tau^9 \\
L7a3 & 20\sigma^9\tau^9 + \sigma^9\tau^3 + 4\sigma^3\tau^9 \\
L7a4 & 20\sigma^9\tau^9 + \sigma^9\tau^3 + 4\sigma^3\tau^9 \\
L7a5 & 9\sigma^9\tau^9 + 4\sigma^3\tau^3 \\
L7a6 & 4\sigma^9\tau^9 + \sigma^9\tau^3 + 4\sigma^3\tau^9 + 4\sigma^3\tau^3 \\
L7a7 & 8\sigma^{27}\tau^{27} + \sigma^{27}\tau^3 + 6\sigma^9\tau^9 + 12\sigma^9\tau^3 + 8\sigma^3\tau^{27} + 6\sigma^3\tau^9 \\
L7n1 & 16\sigma^9\tau^9 + \sigma^9\tau^3 + 8\sigma^3\tau^9 \\
L7n2 & 20\sigma^9\tau^9 + \sigma^9\tau^3 + 4\sigma^3\tau^9
\end{array}
\]

In particular, we note that the links $L5a1$ and $L6a1$ have the same quandle
module polynomial value $21\sigma^9+4\sigma^3$ but are distinguished by their 
quandle module quiver polynomials 
\[
\Phi_{X,f}^{t,s}(L5a1)=20\sigma^9\tau^9+\sigma^9\tau^3+4\sigma^3\tau^9
\ne 21\sigma^9\tau^9+4\sigma^3\tau^9=\Phi_{X,f}^{t,s}(L6a1).
\]
Hence, the quandle module quiver polynomial is not determined by the quandle 
module polynomial.
\end{example}

Our next example shows that the new invariant is not determined by the 
quandle coloring quiver.

\begin{example}
In \cite{CN2}, we find an example of two links $L7n1$ and $L7n2$ which have 
isomorphic quandle coloring quivers with respect to the quandle $X$ and 
endomorphism $f$ given by
\[
\begin{array}{r|rrrr}
\tr & 1 & 2 & 3 & 4 \\ \hline
1 & 1 & 1 & 1 & 1 \\
2 & 4 & 2 & 2 & 2 \\
3 & 3 & 3 & 3 & 3 \\
4 & 2 & 4 & 4 & 4
\end{array}\quad f=[1,3,3,3].
\]
Our \texttt{python} computations reveal that the quandle module quiver 
polynomials for these links with respect to the quandle module over 
$\mathbb{Z}_3$ given by
\[
\begin{array}{r|rrrr}
t & 1 & 2 & 3 & 4 \\ \hline
1 & 1 & 1 & 1 & 1 \\
2 & 1 & 1 & 2 & 1 \\
3 & 1 & 1 & 2 & 1 \\
4 & 1 & 1 & 2 & 1
\end{array}\quad
\begin{array}{r|rrrr}
s & 1 & 2 & 3 & 4 \\ \hline
1 & 0 & 0 & 0 & 0 \\
2 & 1 & 0 & 2 & 0 \\
3 & 1 & 0 & 2 & 0 \\
4 & 1 & 0 & 2 & 0
\end{array}
\]
distinguish the links, with 
\[\Phi_{X,f}^{t,s}(L7n1)
=3\sigma^9\tau^9+8\sigma^9\tau^3+4\sigma^3\tau^9+\sigma^3\tau^9
\ne 
7\sigma^9\tau^9+8\sigma^9\tau^3+\sigma^3\tau^9
=\Phi_{X,f}^{t,s}(L7n2).\]
\end{example}

\begin{example}\label{ex:mod6}
For our next example we selected a quandle $X$, endomorphism $f$ and
module $s,t$ with $\mathbb{Z}_6$ coefficients as below and computed the quandle module
quiver polynomial for prime links with up to seven crossings. The results are 
in the table.
\[f=[1,3,2]\quad 
\begin{array}{r|rrr}
\tr & 1 & 2 & 3 \\ \hline
1 & 1 & 1 & 2\\
2 & 2 & 2 & 1\\ 
3 & 3 & 3 & 3
\end{array}
\quad
\begin{array}{r|rrr}
t & 1 & 2 & 3 \\ \hline
1 & 5 & 5 & 5 \\
2 & 5 & 5 & 5 \\
3 & 5 & 5 & 1 
\end{array}
\quad
\begin{array}{r|rrr}
s & 1 & 2 & 3 \\ \hline
1 & 2 & 4 & 1 \\
2 & 4 & 2 & 5 \\
3 & 0 & 0 & 0
\end{array}
\]
\[
\begin{array}{r|l}
\Phi_{X,f}^{\tau,\sigma}(L) & L \\ \hline
L2a1 & \sigma^{36}\tau^{12} + \sigma^{12}\tau^{36} + \sigma^{12}\tau^{12} + 2\sigma^{12}\tau^2 \\
L4a1 & \sigma^{36}\tau^{12} + 4\sigma^{36}\tau^6 + \sigma^{12}\tau^{36} + 3\sigma^{12}\tau^{12} \\
L5a1 & 2\sigma^{36}\tau^{12} + \sigma^{36}\tau^6 + 2\sigma^{36}\tau^2 + 2\sigma^{12}\tau^{36} + 2\sigma^{12}\tau^{12}\\
L6a1 & 3\sigma^{36}\tau^{36} + 2\sigma^{36}\tau^{12} + 4\sigma^{36}\tau^2\\
L6a2 & \sigma^{36}\tau^{12} + \sigma^{12}\tau^{36} + \sigma^{12}\tau^{12} + 2\sigma^{12}\tau^2\\
L6a3 & 3\sigma^{36}\tau^{36} + 2\sigma^{36}\tau^2\\
L6a4 & 3\sigma^{216}\tau^{72} + 3\sigma^{72}\tau^{216} + 9\sigma^{72}\tau^{72} + 6\sigma^{72}\tau^{36}\\
L6a5 & \sigma^{216}\tau^{72} + \sigma^{72}\tau^{216} + \sigma^{72}\tau^{72} + 6\sigma^{72}\tau^{12} + 6\sigma^{12}\tau^4\\
L6n1 & \sigma^{216}\tau^{24} + \sigma^{24}\tau^{216} + \sigma^{24}\tau^{24} + 6\sigma^{24}\tau^{12} + 3\sigma^{12}\tau^8 + 3\sigma^{12}\tau^4\\
L7a1 & 4\sigma^{36}\tau^{36} + 2\sigma^{36}\tau^{12} + 3\sigma^{36}\tau^6 \\
L7a2 & \sigma^{36}\tau^{12} + \sigma^{36}\tau^6 + 3\sigma^{36}\tau^2 + \sigma^{12}\tau^{36} + 3\sigma^{12}\tau^{12}\\
L7a3 & 2\sigma^{36}\tau^{12} + 2\sigma^{36}\tau^6 + \sigma^{36}\tau^2 + 2\sigma^{12}\tau^{36} + 2\sigma^{12}\tau^{12}\\
L7a4 & \sigma^{36}\tau^{18} + 2\sigma^{36}\tau^{12} + \sigma^{36}\tau^6 + \sigma^{36}\tau^2 + 2\sigma^{12}\tau^{36} + 2\sigma^{12}\tau^{12} \\
L7a5 & 3\sigma^{36}\tau^{36} + 2\sigma^{36}\tau^2\\
L7a6 & \sigma^{36}\tau^{12} + \sigma^{12}\tau^{36} + \sigma^{12}\tau^{12} + 2\sigma^{12}\tau^2 \\
L7a7 & \sigma^{216}\tau^{24} + \sigma^{24}\tau^{216} + \sigma^{24}\tau^{24} + 5\sigma^{24}\tau^{12} + \sigma^{24}\tau^4 + \sigma^{12}\tau^{36} + 3\sigma^{12}\tau^{12} + 2\sigma^{12}\tau^4 \\
L7n1 & \sigma^{36}\tau^{12} + 2\sigma^{36}\tau^4 + 2\sigma^{36}\tau^2 + \sigma^{12}\tau^{36} + 3\sigma^{12}\tau^{12}\\
L7n2 & 3\sigma^{36}\tau^{12} + \sigma^{36}\tau^6 + \sigma^{36}\tau^2 + 2\sigma^{12}\tau^{36} + 2\sigma^{12}\tau^{12}\\
\end{array}
\]
\end{example}

\begin{example}
We note that the combination of quandle, module and endomorphism in Example
\ref{ex:mod6} shows that
the invariant can be nontrivial for knots, even when the set of
quandle colorings seems trivial. Both the trefoild and the unknot
have quandle coloring quiver consisting of one bigon and one loop.
However, the trefoil's three $X$-colorings yield
three bead-coloring matrices over $\mathbb{Z}_6$, two of which are
\[
\left[\begin{array}{rrr}
5 & 2 & 5 \\
5 & 5 & 2 \\
2 & 5 & 5
\end{array}
\right]\stackrel{\mathrm{Row\ Equiv.\ over}\ \mathbb{Z}_6}{\longleftrightarrow}
\left[\begin{array}{rrr}
1 & 1 & 4 \\
0 & 3 & 3 \\
0 & 0 & 0
\end{array}
\right],
\quad |\mathrm{Ker}(M)|=6(3)=18
\] 
and one of which is
\[
\left[\begin{array}{rrr}
1 & 0 & 5 \\
5 & 1 & 0 \\
0 & 5 & 1
\end{array}
\right]\longleftrightarrow
\left[\begin{array}{rrr}
1 & 0 & 5 \\
0 & 1 & 5 \\
0 & 0 & 0
\end{array}
\right],
\quad |\mathrm{Ker}(M)|=6.
\]
Then the quandle module quiver is
\[\includegraphics{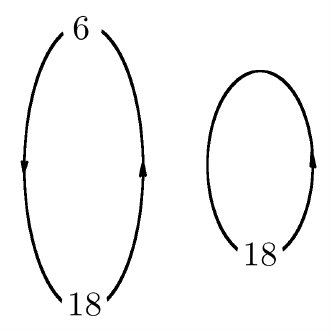}\]
and we obtain quandle module quiver polynomial
\[\Phi_{X,f}^{\tau,\sigma}(3_1)=\sigma^{18}\tau^6+\sigma^{6}\tau^{18}
+\sigma^{18}\tau^{18}.\]
This distinguishes the trefoil from the unknot which has 
quandle module quiver polynomial
\[\Phi_{X,f}^{\tau,\sigma}(0_1)=3\sigma^{6}\tau^6.\]
\end{example}

\section{\large\textbf{Quandle Module Quiver Questions}}\label{QQQ}

We conclude with some questions for future research, noting that this paper 
is only intended to introduce the new structure of quandle module quivers 
and its associated knot invariants -- much remains to be discovered.

What is the precise relationship between quandle module quivers and quandle 
cocycle quivers? A quandle module quiver is a quiver with vertices decorated 
with modules; this setup is almost begging for a Khovanov-style chain complex 
construction. What maps, if any, can be assigned to the arrows in the quiver
in a natural (functorial) way? What is the geometric meaning of these 
algebraic and combinatorially-defined invariants?

\bibliography{ki-sn}{}
\bibliographystyle{abbrv}

\bigskip

\noindent
\textsc{Department of Mathematical Sciences \\
Claremont McKenna College \\
850 Columbia Ave. \\
Claremont, CA 91711} 

\

\end{document}